\documentclass[10pt]{amsart}

\usepackage{amsmath,amsfonts, latexsym,graphicx, footmisc, amssymb, mathtools, enumerate, mdwlist}

\usepackage{hyperref}
\hypersetup{colorlinks=true, urlcolor=blue, citecolor=blue, linkcolor=blue}

\newcommand{\U}{{\mathcal U}}
\newcommand{\0}{{\mathbf 0}}
\newcommand{\C}{{\mathbb C}}
\newcommand{\Z}{{\mathbb Z}}

\newcommand{\Fdot}{\mathbf F^\bullet}

\newcommand{\mult}{{\operatorname{mult}}}

\newcommand{\mpr}{\operatorname{mpr}}

\newtheorem{defn0}{Definition}[section]
\newtheorem{prop0}[defn0]{Proposition}
\newtheorem{conj0}[defn0]{Conjecture}
\newtheorem{thm0}[defn0]{Theorem}
\newtheorem{lem0}[defn0]{Lemma}
\newtheorem{corollary0}[defn0]{Corollary}
\newtheorem{example0}[defn0]{Example}
\newtheorem{remark0}[defn0]{Remark}
\newtheorem{question0}[defn0]{Question}
\newtheorem{exercise0}[defn0]{Exercise}

\newenvironment{defn}{\begin{defn0}}{\end{defn0}}
\newenvironment{prop}{\begin{prop0}}{\end{prop0}}
\newenvironment{conj}{\begin{conj0}}{\end{conj0}}
\newenvironment{thm}{\begin{thm0}}{\end{thm0}}
\newenvironment{lem}{\begin{lem0}}{\end{lem0}}
\newenvironment{cor}{\begin{corollary0}}{\end{corollary0}}

\newenvironment{exm}{\begin{example0}\rm}{\end{example0}}
\newenvironment{rem}{\begin{remark0}\rm}{\end{remark0}}
\newenvironment{ques}{\begin{question0}\rm}{\end{question0}}

\newcommand{\defref}[1]{Definition~\ref{#1}}
\newcommand{\propref}[1]{Proposition~\ref{#1}}
\newcommand{\thmref}[1]{Theorem~\ref{#1}}
\newcommand{\lemref}[1]{Lemma~\ref{#1}}

\newcommand{\exref}[1]{Example~\ref{#1}}
\newcommand{\secref}[1]{Section~\ref{#1}}

\newcommand{\quesref}[1]{Question~\ref{#1}}

\newcommand{\mbf}[1]{{\mathbf #1}}

\title[Minkowski Inequalities  and non-isolated hypersurface singularities]{Minkowski Inequalities and non-isolated hypersurface singularities}



\author{David B. Massey}

\date{}

\begin{document}

\begin{abstract} We derive a number of inequalities involving L\^e numbers of non-isolated hypersurface singularities. In particular, we derive L\^e-Iomdine formulas with inequalities and use these, together with Teissier's Minkowski inequalities for sectional Milnor numbers of isolated hypersurface singularities, to arrive at ``Minkowski inequalities'' for L\^e numbers.
\end{abstract}

\maketitle




\section{Introduction}

Throughout this paper, we use the following notation:  $\U$ is an open neighborhood of the origin in $\C^{n+1}$, $f: (\U,\0)\rightarrow (\C,0)$ is a complex analytic function which is not locally constant at $\0$, $\Sigma f$ is the critical locus of $f$, $\0\in\Sigma f$, and $s:=\dim_\0\Sigma f$. It will be important to note that  $\0\in\Sigma f$ implies that $\mult_\0f\geq 2$.

\smallskip

Since we introduced the L\^e numbers $\lambda^j_{f, \mbf z}$ in \cite{levar1}, we have proved many results about them, but have never looked seriously at inequalities. However, recently, we have been motivated to consider inequalities with L\^e numbers for two reasons: the recent work of June Huh in \cite{huhmilnor} reminded us of the beauty of Teissier's now-classic result on Minkowski inequalities in \cite{teissierappend} (see \secref{sec:prelim}) for sectional Milnor numbers of isolated hypersurface singularities, and a conjecture of Bobadilla (see \secref{sec:questions}) made us wonder if the conjecture could be proved in at least some cases via inequalities on L\^e numbers.

\medskip

Our main tool from obtaining inequalities when $s\geq 1$ are our new L\^e-Iomdine formulas with inequalities in \thmref{thm:leiom}. In a sense, these are not new; they follow from the previous proofs of Theorem 4.5 in \cite{lecycles} with a few modifications. We also improve our prior L\^e-Iomdine equalities by giving a new upper-bound on the {\it maximum polar multiplicity} (see \cite{lecycles} or \defref{def:mpr}) in \propref{prop:newmpr}.

\medskip

Using the L\^e-Iomdine inequalities, we prove in \propref{prop:funbound}:

\medskip

\noindent {\bf Proposition}: {\it Suppose that the L\^e numbers of $f$ at the origin with respect to coordinates $\mbf z$ are defined. Then,

$$
\sum_{j=0}^s (-1+\mult_\0 f)^j\lambda^j_{f, \mbf z}\geq(-1+\mult_\0f)^{n+1},
$$
with equality holding if $f$ is a homogeneous polynomial.}

\medskip

To generalize Teissier's Minkowski inequalities to L\^e numbers and non-isolated singularities, we need to use restrictions to {\bf generic} linear subspaces and then use generic coordinates for calculating the L\^e numbers. This is unusual for us, as we normally like to use coordinates which are ``generic'' in some sense which is easy to check in examples, such as just being generic enough that the L\^e numbers are defined (as above). It also means that, in this paper,  we will use two notations for the L\^e numbers (at $\0$): $\lambda^j_{f, \mbf z}$ when we are going to specify how generic the coordinates $\mbf z$ need to be, and $\lambda^j_f$ when the coordinates are arbitrarily generic (i.e., so generic that all $\lambda^j_f$ attain their generic values at a given point).

We also need to alter Teissier's notation for the sectional Milnor numbers $\mu^{(k)}(f)$. We let $f^{[k]}$ denote the restriction of $f$ to a generic linear subspace of dimension $k$. With this notation, rather than writing $\mu^{(k)}(f)$ for a sectional Milnor number, we will write $\mu(f^{[k]})$. The point is that we can now refer to $\lambda^j_{f^{[k]}}$, where we will always use generic coordinates to accompany the generic linear subspace restriction when $k\neq n+1$. Following Teissier, we define $\lambda^0_{f^{[0]}}=\mu^{(0)}(f)=1$.

\medskip

The inequalities that we obtain when $s=1$ are sharper than what we obtain we $s\geq 2$. So we separate the cases where $s=1$ and where $s\geq 2$ into different sections of this paper. 

\medskip

When $s=1$, we prove in \thmref{thm:mainone}:

\medskip

\noindent{\bf Theorem}: Suppose that $n\geq 1$, $s=1$, and the L\^e numbers of $f$ at the origin are defined. Then we have the following inequality:
$$
\frac{\lambda^0_{f, \mbf z}+\big(\mu(f^{[n]})-\mu(f^{[n-1]})+1\big)\lambda^1_{f, \mbf z}}{\mu(f^{[n]})}\ \geq \ \frac{\mu(f^{[n]})}{\mu(f^{[n-1]})}.
$$

\medskip

Note that if $s=0$, then $\lambda^1_{f, \mbf z}=0$ and $\lambda^0_{f, \mbf z}=\mu(f)$, so that the inequality reduces to Teissier's inequality.

\medskip

The case when $s\geq 2$ is complex enough that we will not state the result in the introduction. Also, there are a number of other interesting, relatively easy, inequalities with L\^e numbers (and relative polar numbers) which we prove along the way.

\medskip

\section{Preliminary definitions and results}\label{sec:prelim}

If $s=0$, then the Milnor number $\mu(f)$ of $f$ at the origin is defined. In \cite{teissierappend}, Teissier defines the sectional Milnor numbers $\mu^{(j)}(f)=\mu(f^{[j]})$ for $0\leq j\leq n+1$ by letting $\mu(f^{[j]})$ equal the Milnor number of $f$ restricted to a generic $j$-dimensional linear subspace of $\C^{n+1}$, where $\mu(f^{[0]})=1$ by definition. Note that $\mu(f^{[1]})=-1+\mult_\0 f$ and $\mu(f^{[n+1]})=\mu(f)$. Teissier shows in  (in a more-general context) that one has the following Minkowski inequalities, i.e., that $\mu(f^{[j]})$ is a log-convex sequence:

\medskip

\begin{thm}\textnormal{(Teissier, \cite{teissierappend}}\textnormal{)}\label{thm:teissier} Suppose that $s=0$. Then the following inequalities hold:

\smallskip

$$\frac{\mu(f^{[n+1]})}{\mu(f^{[n]})}\geq \frac{\mu(f^{[n]})}{\mu(f^{[n-1]})}\geq\cdots\geq \frac{\mu(f^{[2]})}{\mu(f^{[1]})}\geq \frac{\mu(f^{[1]})}{\mu(f^{[0]})}.$$
\end{thm}

\medskip

As an immediate corollary to \thmref{thm:teissier}, there is the well-known:

\medskip

\begin{cor}\label{cor:teissier} Suppose that $s=0$. Then, for all $k$ with $0\leq k\leq n$, 
$$\mu(f^{[k+1]})\geq [(\mult_\0 f)-1]\mu(f^{[k]}) \hskip 0.2in \textnormal{ and so }\hskip 0.2in\mu(f^{[k+1]})\geq [(\mult_\0 f)-1]^{k+1}.$$
\end{cor}

\smallskip

\noindent Note that the second inequality above is also immediate from Corollary 12.4 of \cite{fulton}.

\medskip

Throughout this section and paper, we use the L\^e and relative polar cycles, $\Lambda^j_{f, \mbf z}$ and $\Gamma^j_{f, \mbf z}$, respectively, of $f$ with respect to $\mbf z$  as defined and investigated in \cite{lecycles}. The L\^e and relative polar numbers, $\lambda^j_{f, \mbf z}(\mbf p)$ and $\gamma^j_{f, \mbf z}(\mbf p)$, respectively, of $f$ with respect to $\mbf z$ at the point $\mbf p=(p_0, \dots, p_n)$ are defined by:
$$
\lambda^j_{f, \mbf z}(\mbf p) :=\big(\Lambda_{f, \mbf z}^j\cdot V(z_0-p_0, \dots, z_{j-1}-p_{j-1})\big)_{\mbf p}
$$
(respectively,
$$
\gamma^j_{f, \mbf z}(\mbf p) :=\big(\Gamma_{f, \mbf z}^j\cdot V(z_0-p_0, \dots, z_{j-1}-p_{j-1})\big)_{\mbf p},
$$
provided that, for all $j$,  $\Lambda_{f, \mbf z}^j$ and $\Gamma_{f, \mbf z}^j$ are purely $j$-dimensional and the intersections are proper at $\mbf p$, i.e., $\mbf p$ is an isolated point in the intersection.

In this paper, since we concentrate our attention at the origin (for convenience), when we omit the the point $\mbf p$ from the notation, we mean that $\mbf p=\0$.

When $j=0$, we mean that $\lambda^0_{f, \mbf z}(\mbf p):=\big(\Lambda_{f, \mbf z}^0\big)_{\mbf p}$, i.e., the coefficient of $\mbf p$ in the $0$-dimensional cycle $\Lambda_{f, \mbf z}^0$. Also, note that $\Gamma_{f, \mbf z}^0=0$ and so $\gamma^0_{f, \mbf z}(\mbf p)=0$.

Note that if the L\^e numbers (respectively, relative polar numbers) exist at $\mbf p$, they exist at all points near $\mbf p$.

\medskip

We remind the reader that the generic values of the L\^e numbers, that is, the values for generic $\mbf z$, are {\bf not} necessarily the minimum values of the L\^e numbers. However, recall Corollary 4.16 from \cite{lecycles}, which tells us that the tuple of L\^e numbers is upper-semicontinuous in the lexicographical ordering (from higher-dimensional L\^e number to lower).  This implies that for generic $\mbf z$, $\lambda^s_{f, \mbf z}$ attains its minimum value, and then for generic $\mbf z$ among those, $\lambda^{s-1}_{f, \mbf z}$ attains its minimum value, and in general, if $\Omega$ is an open, dense set of coordinates such that the L\^e numbers are defined for all $\mbf z\in\Omega$, and the higher-dimensional L\^e numbers attain their generic values for all $\mbf z\in \Omega$, then, for $\mbf z\in\Omega$, the generic value of the next-lower L\^e number is its minimum value  for $\mbf z\in\Omega$.

\medskip

Another bit of notation: if we have an effective cycle $\Omega:=\sum m_V V$, where the $V$ are irreducible analytic sets and the $m_V$ are positive integers, then we let the multiplicity of $\Omega$ at a point $p$ equal $\mult_p\Omega=\sum m_V\mult_pV$.

\bigskip

We should also remind the reader that we have the following:

\begin{prop}\label{prop:multint} For generic $\mbf z$, $\gamma^j_{f,\mbf z}$ is equal to the multiplicity of the relative polar cycle $\Gamma^j_{f, \mbf z}$ at the origin, i.e., 
$$\gamma^j_{f,\mbf z}:=\Big(\Gamma^j_{f, \mbf z}\cdot V(z_0, \dots, z_{j-1})\Big)_\0=\mult_\0 \Gamma^j_{f, \mbf z}.$$

 Furthermore, the analogous statement for the L\^e numbers is true, i.e., for generic $\mbf z$, 
 $$\lambda^j_{f,\mbf z}:=\Big(\Lambda^j_{f, \mbf z}\cdot V(z_0, \dots, z_{j-1})\Big)_\0=\mult_\0 \Lambda^j_{f, \mbf z}.$$
\end{prop}
\begin{proof}
The equality for relative polar numbers are proved by Teissier  in Corollaire IV.5.4.3 of \cite{teissiervp2}. (In fact, if the coordinates are generic enough, Teissier proves that $\Gamma^1_{f, \mbf z}$ is reduced, i.e., for all $C$, $m_C=1$; however, we do not need this level of genericity.)

Theorem 7.5 of \cite{numinvar} applied to the case where $\Fdot$ is the shifted complex of vanishing cycles of the constant sheaf along $f$ describes the L\^e cycles in terms of the characteristic cycle of $\Fdot$ and the absolute polar varieties of strata. Now the equality for L\^e numbers also follows Corollaire IV.5.4.3 of \cite{teissiervp2}.
\end{proof}

\smallskip

\begin{rem} The reader should realize that \propref{prop:multint} is far from obvious. It is certainly true that for a fixed cycle $C$, its multiplicity at the origin can be calculated as the intersection number in $\U$ with a generic linear subspace of complementary dimension (see \cite{fulton}, Example 12.4.5). However, $\Gamma^j_{f, \mbf z}$ and $\Lambda^j_{f, \mbf z}$ move as $\mbf z$ changes, and so the proof is highly non-trivial.
\end{rem}

\medskip

\begin{lem}\label{lem:easybound} Suppose that the L\^e numbers of $f$ at the origin with respect to coordinates $\mbf z$ are defined. Consider the relative polar curve as a cycle: $\Gamma^1_{f, \mbf z}=\sum_Cp_C\,C$, where the $C$ are the irreducible curves germs at $\0$. Suppose that $\lambda^0_{f,\mbf z}\neq 0$ or, equivalently, that $\Gamma^1_{f, \mbf z}\neq0$.

Then, 
$$
\lambda^0_{f,\mbf z}\geq \big(\mult_\0 \Gamma^1_{f, \mbf z}\big)\left(\mult_\0 \left(\frac{\partial f}{\partial z_0}\right)\right)\geq \big(\mult_\0\Gamma^1_{f, \mbf z}\big)(-1+\mult_\0 f).
$$
In particular, if $\lambda^0_{f,\mbf z}\neq 0$, then $\lambda^0_{f,\mbf z}\geq -1+\mult_\0 f$.
\end{lem}
\begin{proof} First note that if $\frac{\partial f}{\partial z_0}$ were identically zero, then the relative polar curve $\Gamma^1_{f, \mbf z}$, as a set, would be empty; hence, as a cycle, would be zero. We have excluded this case.

\medskip

Now suppose that $\frac{\partial f}{\partial z_0}$ is not identically zero. Then,
$$
\lambda^0_{f,\mbf z}=\left(\Gamma^1_{f, \mbf z}\cdot V\Big(\frac{\partial f}{\partial z_0}\Big)\right)_\0= \left(\big(\sum_Cp_C\,C\big)\cdot V\Big(\frac{\partial f}{\partial z_0}\Big)\right)_\0.
$$
But it is well-known (and easy for the intersections of curves and hypersurfaces) that
$$
\left(C\cdot V\Big(\frac{\partial f}{\partial z_0}\Big)\right)_\0\geq \big(\mult_\0 C\big)\left(\mult_\0 \left(\frac{\partial f}{\partial z_0}\right)\right);
$$
see Corollary 12.4 of \cite{fulton} for the general statement. This proves the first inequality.

\medskip

The second inequality is also easy. If the coordinates are generic, then we would have
$$
\mult_\0 \frac{\partial f}{\partial z_0} = -1+\mult_\0 f.
$$
However, since we are assuming only that the L\^e numbers are defined, we know only that
$$
\mult_\0 \frac{\partial f}{\partial z_0} \geq -1+\mult_\0 f.
$$
The second inequality follows.
\end{proof}

\medskip

If we are willing to assume that the relative polar numbers of $f$ at the origin with respect to $\mbf z$ are defined, then we can generalize \lemref{lem:easybound}.  :

\begin{prop} Suppose that the L\^e and relative polar numbers at the origin are defined. Then,
$$
\lambda^j_{f, \mbf z}+\gamma^j_{f, \mbf z}\geq(-1+\mult_\0 f)\big(\mult_\0\Gamma^{j+1}_{f, \mbf z}\big).
$$
Thus, with generic coordinates, we have
 $$\lambda^j_f+\gamma^j_f\geq(-1+\mult_\0f)\gamma^{j+1}_f.$$
 
 Furthermore, if $f$ is a homogeneous polynomial, then equality holds in the above inequality for generic coordinates.
\end{prop}
\begin{proof} This is completely straightforward.

By definition,
$$
\Lambda^j_{f, \mbf z}+\Gamma^j_{f, \mbf z}=\Gamma^{j+1}_{f, \mbf z}\cdot V\left(\frac{\partial f}{\partial z_j}\right).
$$
Thus,
$$
\lambda^j_{f, \mbf z}+\gamma^j_{f, \mbf z}=\Big(\Gamma^{j+1}_{f, \mbf z}\cdot V\left(\frac{\partial f}{\partial z_j}\right)\cdot V(z_0, \dots, z_{j-1})\Big)_\0.
$$
But then by Corollary 12.4 of \cite{fulton},
$$
\lambda^j_{f, \mbf z}+\gamma^j_{f, \mbf z}\geq\left(\mult_\0 V\Big(\frac{\partial f}{\partial z_j}\Big)\right)\big(\mult_\0\Gamma^{j+1}_{f, \mbf z}\big), 
$$
which is greater than or equal to $(-1+\mult_\0 f)\big(\mult_\0\Gamma^{j+1}_{f, \mbf z}\big)$.

\medskip

The claim involving equality for homogeneous polynomials follows from Theorem 1.72 of \cite{lesurvey}.
\end{proof}

\medskip

\section{L\^e-Iomdine Formulas with Inequalities}\label{sec:leiomdine}

In this section, we will alter our statement of Theorem 4.5 of \cite{lecycles}, the L\^e-Iomdine formulas,  to include an inequality statement; this essentially follows immediately form the proof in \cite{lecycles}, with just a slight modification. However, we also improve the needed bound on the exponent.

\medskip

We need to discuss briefly our use of not-so-generic for the L\^e numbers in some results; in the next theorem, we let
$\mbf z = (z_0, \dots, z_n)$ be a linear choice of coordinates such that the L\^e numbers of $f$ at the origin are defined. This is a fairly weak condition on $\mbf z$ which just means that all intersections which are involved are proper. It is very convenient in examples to {\bf not} require more genericity of $\mbf z$. 

In addition, the L\^e numbers at the origin being defined implies that each of the relative polar varieties/cycles $\Gamma^j_{f, \mbf z}$ has the correct dimension, i.e., is purely $j$-dimensional. Note that ``purely'' allows for the vacuous case where $\Gamma^j_{f, \mbf z}$ has no irreducible components, that is, is empty/zero. 

However, it is important to realize that the L\^e numbers being defined does {\bf not} imply that the relative polar numbers are defined. Thus, in Item 5 of the following theorem, we have the added hypothesis that $\gamma^1_{f, \mbf z}$ is defined, i.e., that $\dim_\0\big(\Gamma^1_{f, \mbf z}\cap V(z_0))\leq 0$ (where $<0$ would mean $\0$ is not in $\Gamma^1_{f, \mbf z}$). It is worth noting $\gamma^1_{f, \mbf z}(\0)$ being defined  implies that, near $\0\in\Sigma f$, $\Sigma\big(f_{|_{V(z_0)}}\big)=\Sigma f\cap V(z_0)$; see Remark 1.10 in \cite{lecycles}.

\bigskip

Now we need a definition of {\it polar ratios}; we follow \cite{masseysiersma} and\cite{siersmaisoline} (see also \cite{lecycles}, Definition 4.1).

\medskip

\begin{defn}\label{def:mpr} Suppose that  $\Gamma^1_{f, \mbf z}$ is $1$-dimensional at the origin.  

\begin{itemize}

\item Let $C$ be an 
irreducible component of $\Gamma^1_{f, \mbf z}$ (with its reduced structure). Then $C \cap V(z_0)$ is $0$-dimensional at
the origin if and only if $C\cap V(f)$  is $0$-dimensional at the origin, and in this case
the {\bf polar ratio}\index{polar ratio} of $C$ (for $f$ at $\mbf 0$ with respect to $\mbf z$) is $$\frac{(C \cdot V(f))_{{}_\mbf
0}}{(C \cdot V(z_0))_{{}_\mbf 0}} = \frac{\left(C \cdot V\left(\frac{\partial f}{\partial z_0}\right)\right)_{{}_\mbf 0}
+(C \cdot V(z_0))_{{}_\mbf 0}}{(C \cdot V(z_0))_{{}_\mbf 0}} = \frac{\left(C \cdot V\left(\frac{\partial f}{\partial
z_0}\right)\right)_{{}_\mbf 0}}{(C \cdot V(z_0))_{{}_\mbf 0}} + 1 .$$

\medskip

\item Let $C$ be an 
irreducible component of $\Gamma^1_{f, \mbf z}$ such that  $C \cap V(z_0)$ is $1$-dimensional at
the origin, i.e., is contained in $V(z_0)$ near the origin. Then we define the polar ratio of $C$ to be $1$.

\end{itemize}

\medskip

\noindent We will be interested in the {\bf maximum polar ratio} over all of the irreducible components $C$  of $\Gamma^1_{f, \mbf z}$ at the origin; we denote this maximum by $\mpr\{f, \mbf z\}$.

\smallskip

If the set $\Gamma^1_{f, \mbf z}$ is empty at the origin, i.e., if the cycle $\Gamma^1_{f, \mbf z}=0$ at the origin, then we define $\mpr\{f, \mbf z\}$ to be $1$.
\end{defn}

\medskip

\begin{lem}\label{lem:mprmult} Suppose that  the L\^e numbers $\lambda^j_{f, \mbf z}$ are defined. Suppose also that $\mbf z$ is generic enough so that  $\gamma^1_{f, \mbf z}$ is defined and $\gamma^1_{f, \mbf z}=\mult_\0 \Gamma^1_{f, \mbf z}$. Finally, suppose that  $\gamma^1_{f, \mbf z}\neq0$. Then,
$$ \mpr\{f, \mbf z\}\geq \mult_\0 f.$$
\end{lem}
\begin{proof} Let $C$ be one of the irreducible (reduced) components of $\Gamma^1_{f, \mbf z}$ for which the polar ratio is maximal. Then, by Corollary 12.4 of \cite{fulton}, 
$$\left(C \cdot V\left(\frac{\partial f}{\partial z_0}\right)\right)_{{}_\mbf 0}\geq \big(\mult_\0 C\big)\left(\mult_\0 \left(\frac{\partial f}{\partial z_0}\right)\right) = \big((C\cdot V(z_0))_\0\big)\left(\mult_\0 \left(\frac{\partial f}{\partial z_0}\right)\right),
$$
and so we have
$$
\frac{(C \cdot V\left(\frac{\partial f}{\partial z_0}\right))_{{}_\mbf 0}}{(C \cdot V(z_0))_{{}_\mbf 0}} + 1 \geq \left(\mult_\0 \left(\frac{\partial f}{\partial z_0}\right)\right)+1\geq \mult_\0 f.
$$
\end{proof}

We need the following trivial lemma. Even though the proof is via elementary school algebra, the inequality is important for us and is not obvious, so we give the short argument.

\begin{lem}\label{lem:dumbalg1} Suppose that $a$ and $b$ are real numbers, that  $a\geq 1$ and $b\geq 1$. Then,
$$
a\geq b\hskip 0.1in\textnormal{ if and only if }\hskip 0.1in a-b+1\geq \frac{a}{b}.
$$
\end{lem}
\begin{proof}
If $b=1$, then the conclusion is that $a\geq 1$ if and only if $a\geq a$. As we are assuming that $a\geq 1$, this statement is true.

Now suppose that $b>1$. Then,
$$
a\geq b \hskip 0.07in \Leftrightarrow  \hskip 0.07in  a(b-1)\geq b(b-1)  \hskip 0.07in \Leftrightarrow  \hskip 0.07in ab-a\geq b^2-b \hskip 0.07in \Leftrightarrow  \hskip 0.07in ab-b^2+b\geq a \hskip 0.07in \Leftrightarrow  \hskip 0.07in a-b+1\geq \frac{a}{b}.
$$
\end{proof}

\medskip

\begin{prop}\label{prop:newmpr} Suppose that  the L\^e numbers $\lambda^j_{f, \mbf z}$ are defined.

\begin{enumerate}

\item We have the inequality 
$$
\lambda^0_{f, \mbf z}+1\geq  \mpr\{f, \mbf z\}.
$$

\medskip

 \item  Suppose also that $\mbf z$ is generic enough so that   $\gamma^1_{f, \mbf z}$ is defined and $\gamma^1_{f, \mbf z}=\mult_\0 \Gamma^1_{f, \mbf z}$. Use $\check{\mbf z}:=(z_1, \dots, z_n)$ as coordinates on $V(z_0)$. Then,
$$\lambda^0_{f, \mbf z}+\lambda^1_{f, \mbf z}-\lambda^0_{f_{|_{V(z_0)}}, \check{\mbf z}}+2=\lambda^0_{f, \mbf z}-\gamma^1_{f, \mbf z}+2 \geq \mpr\{f, \mbf z\}.$$
\end{enumerate}
\end{prop}
\begin{proof} We work in a neighborhood of the origin.

Item 1 is trivial given the definitions of $\lambda^0_{f, \mbf z}$ and the maximum polar multiplicity.

\medskip

\noindent Proof of Item 2:

\smallskip

First, note that $\lambda^0_{f_{|_{V(z_0)}}, \check{\mbf z}}=\gamma^1_{f, \mbf z}+\lambda^1_{f, \mbf z}$ by Proposition 1.21 of \cite{lecycles}; this explains the equality in Item (2).

\smallskip

Now if $\Gamma^1_{f, \mbf z}=0$, the inequalities become $2\geq 1$.

\medskip

Suppose then that $\Gamma^1_{f, \mbf z}\neq0$, and that $\gamma^1_{f, \mbf z}=\mult_\0 \Gamma^1_{f, \mbf z}$.

\smallskip

Let $C$ be one of the irreducible (reduced) components of $\Gamma^1_{f, \mbf z}$ for which the polar ratio is maximal. As $\left(C\cdot V(z_0)\right)_\0=\mult_\0 C$, we have that $\left(C\cdot V\left(\frac{\partial f}{\partial z_0}\right)\right)_\0\geq \left(C\cdot V(z_0)\right)_\0\geq 1$ and, hence, we may apply \lemref{lem:dumbalg1} to conclude that 
$$
\left(C\cdot V\left(\frac{\partial f}{\partial z_0}\right)\right)_\0- \left(C\cdot V(z_0)\right)_\0 + 2\geq \mpr\{f, \mbf z\}.
$$
Thus, {\it a fortiori}, $\lambda^0_f-\gamma^1_f+2 \geq \mpr\{f, \mbf z\}$.
\end{proof}

\medskip

Our proofs of Lemma 4.3 and Theorem 4.5 from \cite{lecycles}  prove the following result with very little modification; we simply did not have a use for the inequalities before now. However, for clarity, we will explain the mild change required, including why we need for $m$ to be large only to replace the inequality in Item 4 with equality in Item 5.

\medskip

\begin{thm}\label{thm:leiom} Let $m \geqslant 2$ and let $f:(\mathcal U, \mbf 0) \rightarrow (\Bbb C, 0)$ be an analytic function, and let $s:=\dim_\0\Sigma f$. Let
$\mbf z = (z_0, \dots, z_n)$ be a linear choice of coordinates such that the L\^e numbers of $f$ at the origin are defined.  Let $a$ be a non-zero complex number, and use the coordinates $\tilde{\mbf z} = (z_1, \dots, z_n, z_0)$ for $f+ az_0^m$. Use $\check{\mbf z}:=(z_1, \dots, z_n)$ as coordinates on $V(z_0)$

\vskip .1in

Then, for all but a finite number of complex $a$, the following hold.

\medskip

\begin{enumerate}
\item $\Sigma(f +
az_0^m) = \Sigma f \cap V(z_0)$ as germs of sets at $\mbf 0$,  

\medskip

\item If $s\geq 1$, then ${\operatorname{dim}}_\mbf 0\Sigma(f + az_0^m) = s -1$.

\medskip

\item The L\^e numbers of $f + az_0^m$ at the origin, with respect to $\tilde{\mbf z}$,  exist.

\medskip

\item For $1
\leqslant j \leqslant s-1$, $$\lambda^j_{f+ az_0^m, \tilde{\mbf z}} = (m-1)\lambda^{j+1}_{f, \mbf z},$$  
and 
$$\lambda^0_{f+az_0^m,
\tilde{\mbf z}} \leq \lambda^0_{f, \mbf z} + (m-1)\lambda^1_{f, \mbf z};$$

\medskip

\item Furthermore, if $m\geq \mpr\{f, \mbf z\}$ (and $m\geq 2$), then equality holds in the inequality above, i.e., 
$$\lambda^0_{f+az_0^m,
\tilde{\mbf z}} = \lambda^0_{f, \mbf z} + (m-1)\lambda^1_{f, \mbf z}.$$

\smallskip

\noindent In particular, the equality holds if:

\smallskip
\begin{itemize}
\item  $\lambda^0_{f, \mbf z}=0$ and $m\geq 2$; or
\item $\lambda^0_{f, \mbf z}\neq0$ and $m\geq 1+\lambda^0_{f, \mbf z}$; or
\item $\gamma^1_{f, \mbf z}$ is defined, $\gamma^1_{f, \mbf z}=\mult_\0 \Gamma^1_{f, \mbf z}$, and $m\geq \lambda^0_{f, \mbf z}+\lambda^1_{f, \mbf z}-\lambda^0_{f_{|_{V(z_0)}}, \check{\mbf z}}+2$.
\end{itemize}

\bigskip

\item If $\gamma^1_{f, \mbf z}$ is defined, then we have also $\lambda^0_{f+az_0^m,\tilde{\mbf z}} \leq (m-1)\lambda^0_{f_{|_{V(z_0)}}, \check{\mbf z}}$.
\end{enumerate}
\end{thm}
\begin{proof}  The proof of this is identical to that of Theorem 4.5 of \cite{lecycles} once we prove that Items (i), (iii) and (iv) of Lemma 4.3 of \cite{lecycles} hold and that Item (ii) of that lemma holds with the equality replaced by inequality {\bf without assuming anything about the size of $m$ other than $m\geq 2$}. In fact, the proofs of Items (iii) and (iv) from the lemma do not change once we prove Item (i).

\medskip 

\noindent Proof of Lemma 4.3, Item (i) from \cite{lecycles} for $m\geq 2$:

\medskip

In fact, when writing Lemma 4.3, we omitted the trivial case where $\Gamma^1_{f, \mbf z}=0$ ; Item (i) is not technically true in that case. However, $\Gamma^1_{f, \mbf z}=0$ makes all of the later consequences of Item (i) hold trivially.

So assume that $\Gamma^1_{f, \mbf z}=\sum p_C[C]\neq 0$; as the L\^e numbers are defined, this is $1$-dimensional. We want to show that, for a given $C$, for all but a finite number of $a$, $\dim_\0\left(C\cap V\left(\frac{\partial f}{\partial z_0}+maz_0^{m-1}\right)\right)=0$, i.e., near $\0$, $C\not\subseteq V\left(\frac{\partial f}{\partial z_0}+maz_0^{m-1}\right)$. Suppose $\hat a\neq \check a$ and 
$$
C\subseteq V\left(\frac{\partial f}{\partial z_0}+m\hat az_0^{m-1}\right)\cap V\left(\frac{\partial f}{\partial z_0}+m\check az_0^{m-1}\right).
$$
Then $C$ would be contained in $V(z_0)$ and so in  $V\left(\frac{\partial f}{\partial z_0}\right)$; this would contradict that $\lambda^0_{f, \mbf z}$ is defined. Therefore, for each $C$, there is at most one value of $a$ that needs to be avoided. The result follows.

\bigskip

\noindent Proof of Items (4), (5), and $(6)$ for $m\geq 2$:

\medskip

The equalities in Items (4) and (5) are proved exactly as in the proof of Theorem 4.5 of \cite{lecycles}, when combined with \propref{prop:newmpr}. It is the inequalities in (4) and (6) that we need to address.

The proof of Theorem 4.5 of \cite{lecycles} shows that
$$
\lambda^0_{f+az_0^m,
\tilde{\mbf z}} = \left(\Gamma^1_{f, \mbf z}\cdot V\left(\frac{\partial f}{\partial z_0}+maz_0^{m-1}\right)\right)_\0 + (m-1)\lambda^1_{f, \mbf z}.
$$
Letting $\Gamma^1_{f, \mbf z} =\sum_C p_C[C]$ and proceeding as in Theorem 4.5 of \cite{lecycles}, we have, for all but a finite number of $a$,
$$\left(\Gamma^1_{f, \mbf z}\cdot V\left(\frac{\partial f}{\partial z_0}+maz_0^{m-1}\right)\right)_\0=\sum_C p_C \left(C\cdot V\left(\frac{\partial f}{\partial z_0}+maz_0^{m-1}\right)\right)_\0 =
$$
$$
\sum_Cp_C\min\left\{\left(C\cdot V\left(\frac{\partial f}{\partial z_0}\right)\right)_\0,\ (m-1)\left(C\cdot V\left(z_0\right)\right)_\0\right\}.
$$
This last summation equals $\lambda^0_{f, \mbf z}$ if $m\geq \mpr\{f, \mbf z\}$; however, in general, we know that
$$
\sum_Cp_C\min\left\{\left(C\cdot V\left(\frac{\partial f}{\partial z_0}\right)\right)_\0,\ (m-1)\left(C\cdot V\left(z_0\right)\right)_\0\right\}\leq \min\{\lambda^0_{f, \mbf z}, (m-1)\gamma^1_{f, \mbf z}\}.
$$
This proves Item (5). It also proves Item (6) since $(m-1)\gamma^1_{f, \mbf z}+(m-1)\lambda^1_{f, \mbf z} = (m-1)\lambda^0_{f_{|_{V(z_0)}}, \check{\mbf z}}$; see Proposition 1.21 of \cite{lecycles}.
\end{proof}

\bigskip

\begin{rem}\label{rem:leiom} It will be important to us in the next section that \thmref{thm:leiom} applies to the case where $\dim_0\Sigma f=0$. In this case, for $j\geq 1$, $\lambda^j_{f, \mbf z}=0$, and $\lambda^0_{f, \mbf z}=\mu(f)$ the Milnor number of $f$ at the origin; if we let $m:=1+\mu(f)$, then we have that, for all but a finite number of $a$, $f+az_0^m$ has an isolated critical point at $\0$ and $\mu(f+az_0^m)=\mu(f)$.
\end{rem}

\medskip

Before we get to versions of the Minkowski inequalities in the next section, we have the following easy to obtain bound:

\begin{prop}\label{prop:funbound} Suppose that the L\^e numbers of $f$ at the origin with respect to coordinates $\mbf z$ are defined. Then,

$$
\sum_{j=0}^s (-1+\mult_\0 f)^j\lambda^j_{f, \mbf z}\geq(-1+\mult_\0 f)^{n+1},
$$
with equality holding if $f$ is a homogeneous polynomial.
\end{prop}
\begin{proof} Let $m:=\mult_\0 f$. Then, for all but a finite number of $a_j$,
$$
g:=f+a_0z_0^m+a_1z_1^m+\cdots+a_{s-1}z_{s-1}^m
$$
has multiplicity $m$, has a $0$-dimensional critical locus, and \thmref{thm:leiom} may be applied inductively to conclude that 
$$
\sum_{j=0}^s (-1+m)^j\lambda^j_{f, \mbf z}\geq \mu(g)\geq (-1+m)^{n+1}.
$$

The equality for homogeneous polynomials is Corollary 4.7 of \cite{lecycles}.
\end{proof}

\medskip

\begin{exm}\label{exm:bn0} 
\medskip

Consider $f(x,y,z)=(x^2-z^2+y^2)(x-z)$, and use $\mbf z:=(x,y,z)$ for our coordinate system. Then $\Sigma f=V(x-z,y)$, and we find:
$$
\Gamma^2_{f, \mbf z}=V\left(\frac{\partial f}{\partial z}\right)= V\big(2z(x-z)+(x^2-z^2+y^2)\big),
$$

$$
\Gamma^2_{f, \mbf z}\cdot V\left(\frac{\partial f}{\partial y}\right)=V\big(2z(x-z)+(x^2-z^2+y^2)\big)\cdot V(y(x-z)) = 
$$
$$
V\big(y, 2z(x-z)+(x^2-z^2)\big) + 2V(x-z, y)= V(y, 3z+x) +3V(x-z, y)=\Gamma^1_{f, \mbf z}+\Lambda^1_{f, \mbf z},
$$
and
$$
\Gamma^1_{f, \mbf z}\cdot V\left(\frac{\partial f}{\partial x}\right)=V(y, 3z+x)\cdot V(2x(x-z)+x^2-z^2+y^2)=
$$
$$
V(y, 3z+x, 2x(x-z)+x^2-z^2)= V(y, 3z+x, 2x+x+z)+V(y, 3z+x, x-z)=2[\{\0\}].
$$

\medskip

\noindent Thus, $\lambda^1_{f, \mbf z}=3$ and $\lambda^0_{f, \mbf z}=2$. 

\smallskip

Therefore, the equality guaranteed by \propref{prop:funbound} reads
$$
2+(-1+3)\cdot 3=(-1+3)^3.
$$
\end{exm}

\medskip

\begin{exm}\label{exm:whitneyumb1} Consider $f(t, x, y)=y^3-x^4-t^2x^2$, where we use the coordinates $\mbf z:=(t,x,y)$. One easily finds that $\Sigma f=V(x,y)$.

Now, we calculate, as cycles:

$$
\Gamma^2_{f, \mbf z}=V\left(\frac{\partial f}{\partial y}\right)= V(3y^2)=2[V(y)],
$$
$$
\Gamma^2_{f, \mbf z}\cdot V\left(\frac{\partial f}{\partial x}\right)=2V(y, -4x^3-2t^2x)=2V(2x^2+t^2, y)+2V(x,y) = \Gamma^1_{f, \mbf z}+\Lambda^1_{f, \mbf z},
$$
and
$$
\Gamma^1_{f, \mbf z}\cdot V\left(\frac{\partial f}{\partial t}\right)=2V(2x^2+t^2, y)\cdot V(tx^2)= 2V(2x^2+t^2, y)\cdot V(t)+2V(2x^2+t^2, y)\cdot V(x^2)= 12[\{\0\}].
$$
Thus, $\lambda^1_f=2$ and $\lambda^0_f=12$.

\smallskip

The inequality in \propref{prop:funbound} tells us that
$$
16= 12+(-1+3)\cdot2\geq (-1+3)^3=8.
$$

\end{exm}

\section{Minkowski Inequalities: the  $1$-dimensional case}

\medskip

We continue with the notation for the introduction: $\U$ is an open neighborhood of the origin in $\C^{n+1}$,  $f: (\U,\0)\rightarrow (\C,0)$ is a complex analytic function which is nowhere locally constant. We assume that $\0\in\Sigma f$.

\medskip

Before we state and prove a Minkowski inequality, we remind the reader of the notation from the introduction: $f^{[k]}$ denotes the restriction of $f$ to a generic linear subspace of dimension $k$. 

We will also need to use that $(f+az_0^m)^{[k]}=f^{[k]}+az_0^m$. Perhaps this seems obvious to the reader, but, if not, we wish to give a bit of an intuitive explanation which may help in thinking about it. Consider the case where $k=n$, i.e., where we are restricting to a generic hyperplane $H$ through the origin. As we are assuming that $H$ is generic, we may take 
$$H=V(z_1-b_2z_2-b_3z_3-\dots-b_nz_n-b_0z_0)$$
 for a generic choice of the complex numbers $b_k$. This means that,  in the functions $f$ and $f+az_0^m$, we replace $z_1$ with $b_2z_2+b_3z_3+\dots+b_nz_n+b_0z_0$; clearly then $(f+az_0^m)_{|_{H}}=f_{|_{H}}+az_0^m$. One continues in this fashion with each successive hyperplane slice, next replacing $z_2$ with a generic linear combination of $z_3, \dots, z_n, z_0$ and so on. 

\medskip

With this discussion out of the way, we may now use the Minkowski inequalities for isolated singularities, combined with the L\^e-Iomdine formulas (with inequalities), to prove that a Minkowski-type inequality holds when $\dim_\0\Sigma f=1$. First, however, we need to address the case where $n=1$, i.e., where $f$ defines a non-reduced plane curve. In order to interpret the theorem below and its proof in this case, we remind the reader that, by definition, $\mu^{(0)}(f)=\mu(f^{[0]})=1$; we also need to point out that $\gamma^1_{f^{[1]}}=1 =\mu(f^{[0]})$.

\smallskip

\begin{thm}\label{thm:mainone} Suppose that $n\geq 1$, $s=1$, and the L\^e numbers of $f$ at the origin are defined. Then the following inequality holds:
$$
\frac{\lambda^0_{f, \mbf z}+\big(\mu(f^{[n]})-\mu(f^{[n-1]})+1\big)\lambda^1_{f, \mbf z}}{\mu(f^{[n]})}\ \geq \ \frac{\mu(f^{[n]})}{\mu(f^{[n-1]})}.
$$
\end{thm}
\begin{proof} Let $\tilde{\mbf z} = (z_1, \dots, z_n, z_0)$ and  let $a\neq0$ be such that Items (1)-(4) of  \thmref{thm:leiom} hold for $f$, for $f^{[n]}$, and for $f^{[n-1]}$.

Then we have
$$
\frac{\lambda^0_{f,\mbf z}+(m-1)\lambda^1_{f, \mbf z}}{\mu(f^{[n]}+az_0^m)}\geq\frac{\mu(f+az_0^m)}{\mu(f^{[n]}+az_0^m)}\geq \frac{\mu(f^{[n]}+az_0^m)}{\mu(f^{[n-1]}+az_0^m)}\geq \frac{\mu(f^{[n]}+az_0^m)}{\mu(f^{[n-1]})}
$$
 where the middle inequality follows from Teissier's \thmref{thm:teissier}, and the other two inequalities follow from the inequality in Item (4) of \thmref{thm:leiom}.

 As we use generic coordinates for $f^{[n]}$, $\gamma^1_{f^{[n]}}$ is defined and  $\gamma^1_{f^{[n]}}=\mult_\0\Gamma^1_{f^{[n]}}$.  Now, applying Item (5) of \thmref{thm:leiom}, the inequality follows by letting $m=\lambda^0_{f^{[n]}}+\lambda^1_{f^{[n]}}-\lambda^0_{f^{[n-1]}}+2$ and noting that, since $f^{[n]}$ has an isolated critical point at $\0$, we have that $\lambda^1_{f^{[n]}}=0$, $\lambda^0_{f^{[n]}}=\mu(f^{[n]})$, and $\lambda^0_{f^{[n-1]}}=\mu(f^{[n-1]})$.
\end{proof}

\medskip

\begin{rem} The reader should note that the inequality that is concluded in \thmref{thm:mainone} collapses to Teissier's inequality if $\dim_\0\Sigma f=0$, since then $\lambda^1_{f,\mbf z}=0$. Thus \thmref{thm:mainone} holds for $\dim_\0\Sigma f\leq 1$.
\end{rem}

\medskip

One might wonder, in analogy to the isolated critical point case, if we always have $\lambda^0_f\geq \mu(f^{[n]})$. The answer is ``no''. Counterexamples to this ``hope'' are not hard to find; in fact, we looked at one earlier:

\begin{exm}\label{exm:bn02nd} Let us return to \exref{exm:bn0}, where $f(x,y,z)=(x^2-z^2+y^2)(x-z)$ and  $\mbf z=(x,y,z)$. There we found $\lambda^1_{f, \mbf z}=3$ and $\lambda^0_{f, \mbf z}=2$.

Now, as $f$ is homogeneous degree 3 with a 1-dimensional critical locus, $f^{[2]}$ and $f^{[1]}$ have isolated critical points and are also homogeneous of degree 3; thus  $\mu(f^{[2]})=(3-1)^2=4$ and $\mu(f^{[1]})=(3-1)=2$. Hence, in this example $\mu(f^{[2]})>\lambda^0_{f, \mbf z}$.

\medskip

For this example, the inequality in \thmref{thm:mainone} gives us:
$$
\frac{11}{4}=\frac{2+(4-2+1)\cdot 3}{4}=\frac{\lambda^0_{f, \mbf z}+\big(\mu(f^{[2]})-\mu(f^{[1]})+1\big)\lambda^1_{f, \mbf z}}{\mu(f^{[2]})}\ \  \geq \ \frac{\mu(f^{[2]})}{\mu(f^{[1]})}=\frac{4}{2}=2.
$$
\medskip

The reader may wonder if the fact that $\mu(f^{[2]})>\lambda^0_{f, \mbf z}$ in this example is due to the fact that our coordinates for calculating the L\^e numbers were not generic enough, since we required only that the L\^e numbers were defined. Is it possible in this example that, had we chosen our coordinates $\mbf z$ really generic, that we would have $\lambda^0_{f, \mbf z}\geq \mu(f^{[2]})$? 

In fact, $\lambda^1_{f, \mbf z}=3$ and $\lambda^0_{f, \mbf z}=2$ {\bf are} the generic values. To see this, note that $f_{|_{V(x)}}$ and $f_{|_{V(x,y)}}$ have isolated critical points at the origin; thus, the coordinates are {\bf prepolar} as defined in \cite{lecycles}. Therefore, by Theorem 3.3 of \cite{lecycles}, $b_2(f)-b_1(f)=\lambda^0_{f, \mbf z}-\lambda^1_{f, \mbf z}$, where $b_2(f)$ and $b_1(f)$ are the Betti numbers in degrees 2 and 1, respectively, of the Milnor fiber of $f$ at the origin. Thus, if $\lambda^1_{f, \mbf z}$ has its generic value, then so does $\lambda^0_{f, \mbf z}$. But for $\lambda^1_{f, \mbf z}$ to obtain its generic value, it is enough for the L\^e numbers to be defined and for each irreducible component $C$ of $\Sigma f$ at $\0$ to have the intersection number $(C\cdot V(x))_\0$ equal to $\operatorname{mult}_\0 C$; this common value is 1 in the present example.

\medskip

\begin{exm}\label{exm:whitneyumb2} Let us return to \exref{exm:whitneyumb1} where we considered $f(t, x, y)=y^3-x^4-t^2x^2$ and used coordinates $\mbf z:=(t,x,y)$. We found that $\lambda^1_{f, \mbf z}=2$ and $\lambda^0_{f, \mbf z}=12$. 

Now, $\mu(f^{[1]})=-1+\mult_\0 f= -1+3=2$. One either checks by hand or by computer-algebra system (CAS) that $\mu(f^{[2]})=6$.

Now the inequality in \thmref{thm:mainone} gives us:
$$
\frac{11}{3}=\frac{12+(6-2+1)\cdot 2}{6}=\frac{\lambda^0_{f, \mbf z}+\big(\mu(f^{[2]})-\mu(f^{[1]})+1\big)\lambda^1_{f, \mbf z}}{\mu(f^{[2]})}\ \  \geq \ \frac{\mu(f^{[2]})}{\mu(f^{[1]})}=\frac{6}{2}=3.
$$

\end{exm}

\bigskip

\end{exm}

\bigskip

\medskip

\section{Minkowski Inequalities: higher dimensions}

\medskip The relative polar numbers do {\bf not} have nice L\^e-Iomdine equalities we would need to inductively use the expression $\mu(f^{[n]})-\gamma^1_{f^{[n]}}+2$ for an upper-bound on the maximum polar multiplicity of $f^{[n]}$. Consequently, we will use the other upper-bound from \thmref{thm:leiom}: if $\lambda^0_{f^{[n]}}\neq0$ and $1+\lambda^0_{f^{[n]}}\geq \mpr\{f^{[n]}, \hat{\mbf z}\}$ in our induction (where $\hat{\mbf z}$ is a set of generic coordinates for $f^{[n]}$.

\begin{thm}\label{thm:mainmany} Suppose that $n\geq 1$ and the L\^e numbers of $f$ at the origin with respect to $\mbf z$ are defined. Further, suppose that $\lambda^0_{f^{[n]}}\neq0$. Then,

\medskip

$\displaystyle\frac{\lambda^0_{f, \mbf z}+k_1\lambda^1_{f, \mbf z}+k_1k_2\lambda^2_{f, \mbf z}+\cdots + k_1k_2\dots k_s\lambda^s_{f, \mbf z}}{\lambda^0_{f^{[n]}}+k_1\lambda^1_{f^{[n]}}+k_1k_2\lambda^2_{f^{[n]}}+\cdots + k_1k_2\dots k_{s-1}\lambda^{s-1}_{f^{[n]}}}\geq\hfill$

\medskip

$\displaystyle\hfill\frac{\lambda^0_{f^{[n]}}+k_1\lambda^1_{f^{[n]}}+k_1k_2\lambda^2_{f^{[n]}}+\cdots + k_1k_2\dots k_{s-1}\lambda^{s-1}_{f^{[n]}}}{\lambda^0_{f^{[n-1]}}+k_1\lambda^1_{f^{[n-1]}}+k_1k_2\lambda^2_{f^{[n-1]}}+\cdots + k_1k_2\dots k_{s-2}\lambda^{s-2}_{f^{[n-1]}}}$

\medskip

\noindent that is,
$$
\frac{\lambda^0_{f, \mbf z}+\sum_{i=1}^s\big[\lambda^i_{f, \mbf z}\big(\prod_{\ell=1}^i k_\ell\big)\big]}{\lambda^0_{f^{[n]}}+\sum_{i=1}^{s-1}\big[\lambda^{i}_{f^{[n]}}\big(\prod_{\ell=1}^i k_\ell\big)\big]}\geq\frac{\lambda^0_{f^{[n]}}+\sum_{i=1}^{s-1}\big[\lambda^{i}_{f^{[n]}}\big(\prod_{\ell=1}^i k_\ell\big)\big]}{\lambda^0_{f^{[n-1]}}+\sum_{i=1}^{s-2}\big[\lambda^{i}_{f^{[n-1]}}\big(\prod_{\ell=1}^i k_\ell\big)\big]},
$$
where
\begin{itemize}
\item $k_1=\lambda^0_{f^{[n]}}$,
\item $k_2=\lambda^0_{f^{[n]}}+k_1\lambda^1_{f^{[n]}}$,
\item $k_3=\lambda^0_{f^{[n]}}+k_1\lambda^1_{f^{[n]}}+k_1k_2\lambda^2_{f^{[n]}}$, and generally
\item $k_p=\lambda^0_{f^{[n]}}+\sum_{i=1}^{p-1}\big[\lambda^{i}_{f^{[n]}}\big(\prod_{\ell=1}^i k_\ell\big)\big]$,  for $1\leq p\leq s$.
\end{itemize}
\end{thm}
\begin{proof} This immediate from an inductive application of \thmref{thm:leiom} using $m=1+\lambda^0_{f^{[n]}, \mbf z}$, exactly as in Corollary 4.6 of \cite{lecycles} (except there, we did not select the $j_k$'s to be as small as possible).
\end{proof}
\medskip

\begin{rem} We have several comments to make about \thmref{thm:mainmany}.

\smallskip

First, note that the denominator on the left, which equals the numerator on the right, is equal to $k_{s}$.

\smallskip

Second, note that $\lambda^0_{f^{[n]}}= \gamma^1_f+\lambda^1_f$, $\lambda^0_{f^{[n-1]}}= \gamma^2_f+\lambda^2_f$ and, for $i\geq 1$, $\lambda^i_{f^{[n]}}= \lambda^{i+1}_f$ and $\lambda^i_{f^{[n-1]}}= \lambda^{i+2}_f$.

\smallskip

Third, downward induction on $n$ yields the further descending inequalities as in Teissier's sequence of inequalities.

\smallskip 

Finally, suppose that $\lambda^0_{f^{[n]}}=0$; by  Proposition 1.21 of \cite{lecycles}, this is equivalent to $\lambda^0_f=\lambda^1_f=0$ (note the generic coordinates). Let $\omega$ be the largest dimension such that $\lambda^0_{f^{[\omega]}}\neq0$. As $\0\in \Sigma f$ and $\lambda^0_{f^{[1]}}=-1+\mult f$, we know that $\omega\geq 1$. Then one can replace $f$ in  \thmref{thm:mainmany} by $f^{[1+\omega]}$ and replace $n$ by $\omega$. Then one can apply Proposition 1.21 of \cite{lecycles} to obtain inequalities for the generic L\^e numbers of the original $f$:  for $0\leq j\leq n-\omega$, $\lambda^j_f$  would be zero and, for $j\geq n-\omega+1$, $\lambda^{j}_f=\lambda^{j-n+\omega}_{f^{[1+\omega]}}$.
\end{rem} 

\medskip

Of course, examples get more complicated for higher-dimensional critical locus.

\medskip

\begin{exm} Let $f(w,x,y,z)=z^2+(w^4+x^3+y^2)^2$ and use $\mbf z:=(w,x,y,z)$ for coordinates. Then, one easily finds that $\Sigma f=V(z, w^6+x^4+y^3)$, and so $s=2$. By hand or by (CAS), one finds $\lambda^0_{f, \mbf z}=14$, $\lambda^1_{f, \mbf z}=3$, and $\lambda^2_{f, \mbf z}=2$.

\smallskip

Now, $f^{[3]}(w,x,y)= (aw+bx+cy)^2+(w^4+x^3+y^2)^2$ for generic $(a,b,c)\in\C^3$ and $\Sigma(f^{[3]})=V(aw+bx+cy, w^4+x^3+y^2)$. One calculates $\lambda^0_{f^{[3]}}=5$ and $\lambda^1_{f^{[3]}}=2$.

\smallskip

Finally, one calculates $\lambda^0_{f^{[2]}}=\mu(f^{[2]})=3$.

\medskip

\thmref{thm:mainmany} tells us that

$$\frac{\lambda^0_{f, \mbf z}+k_1\lambda^1_{f, \mbf z}+k_1k_2\lambda^2_{f, \mbf z}}{\lambda^0_{f^{[3]}}+k_1\lambda^1_{f^{[3]}}}\geq \frac{\lambda^0_{f^{[3]}}+k_1\lambda^1_{f^{[3]}}}{\lambda^0_{f^{[2]}}},
$$

\medskip

\noindent where $k_1=\lambda^0_{f^{[3]}}=5$ and $k_2=\lambda^0_{f^{[3]}}+k_1\lambda^1_{f^{[3]}}=5+5\cdot 2=15$.

\smallskip

Thus, our ``Minkowski inequality'' becomes
$$
\frac{179}{15}=\frac{14+5\cdot3+5\cdot 15\cdot 2}{5+5\cdot 2}\geq\frac{5+5\cdot 2}{3}=5.
$$

\end{exm}

\medskip

\section{Questions and remarks}\label{sec:questions}

When we first began writing this paper, in the case where $s=1$, we used our classic L\^e-Iomdine formulas with $\lambda^0_{f+az_0^j,
\tilde{\mbf z}}(\mbf 0) = \lambda^0_{f, \mbf z}(\mbf 0) + (j-1)\lambda^1_{f, \mbf z}(\mbf 0)$, if $j\geq 1+\lambda^0_{f, \mathbf z}(\0)$. This gave us two cases.

\medskip

\noindent {\bf Case 1}: Suppose that $\mu(f^{[n]})\geq \lambda^0_{f, \mbf z}$. Then we obtained an inequality that follows trivially from \thmref{thm:mainone}, namely that
$$
\frac{\lambda^0_{f, \mbf z}}{\mu(f^{[n]})}+\lambda^1_{f, \mbf z}\ \geq \ \frac{\mu(f^{[n]})}{\mu(f^{[n-1]})}.
$$

\bigskip

\noindent {\bf Case 2}: Suppose that $\lambda^0_{f, \mbf z}\geq \mu(f^{[n]})$. Then we obtained the inequality
$$
\frac{\lambda^0_{f, \mbf z}}{\mu(f^{[n]})}\big(1+\lambda^1_{f, \mbf z}\big)\ \geq \ \frac{\mu(f^{[n]})}{\mu(f^{[n-1]})}. \leqno{(\dagger)}
$$
Of course, if $\lambda^0_{, \mbf z}> \mu(f^{[n]})$, then the inequality in Case 1 is a stronger inequality, i.e., closer to being an equality, and implies the inequality in ($\dagger$).

On the other hand, if $\mu(f^{[n]})> \lambda^0_{f, \mbf z}$, then the inequality in ($\dagger$) is stronger than the other inequality, but {\bf does the inequality ($\dagger$) always hold}?

The answer is: certainly not. By taking a cross-product family (or a $\mu$-constant family) such as $f(x,y,z)=y^2+z^2$, we have a example where $\mu(f^{[n]})> \lambda^0_{f, \mbf z}=0$, and so ($\dagger$) cannot hold.

\medskip

{\bf The interesting question is}: 

\medskip

\begin{ques}\label{ques:otherinequal} Suppose that $s=1$. Is there an example where 
$$\mu(f^{[n]})> \lambda^0_{f, \mbf z}>0\hskip 0.2in \textnormal{ and }\hskip 0.2in \frac{\lambda^0_{f, \mbf z}}{\mu(f^{[n]})}\big(1+\lambda^1_{f, \mbf z}\big)\ < \ \frac{\mu(f^{[n]})}{\mu(f^{[n-1]})}?$$

\medskip

Thus far, we have found no such example, but we are not ready to conjecture that ($\dagger$) always holds.
\end{ques}

One might search for homogeneous examples that give an affirmative answer to the above question, but the following proposition tells us there are no such examples.

\medskip

\begin{prop}\label{prop:homo} Suppose that $s=1$, the L\^e numbers exist, $\lambda^0_{f, \mbf z}\neq 0$, and $f$ is a homogeneous polynomial. Then,
$$
\frac{\lambda^0_{f, \mbf z}}{\mu(f^{[n]})}\big(1+\lambda^1_{f, \mbf z}\big)\ \geq \ \frac{\mu(f^{[n]})}{\mu(f^{[n-1]})}.
$$
\end{prop}
\begin{proof} Suppose to the contrary that 
$$
\frac{\lambda^0_{f, \mbf z}}{\mu(f^{[n]})}\big(1+\lambda^1_{f, \mbf z}\big)\ < \ \frac{\mu(f^{[n]})}{\mu(f^{[n-1]})}. \leqno{(*)}
$$
We shall derive a contradiction.

\medskip

Let $d\geq 2$ equal the degree of $f$, where $d$ is necessarily at least 2. We know that $\mu(f^{[n]})=(d-1)^n$ and $\mu(f^{[n-1}])=(d-1)^{n-1}$. Furthermore, Corollary 4.7 of \cite{lecycles} tells us that $\lambda^0_{f, \mbf z}=(d-1)^{n+1}-(d-1)\lambda^1_{f, \mbf z}$. Thus $\lambda^0_{f, \mbf z}\neq 0$ implies that $(d-1)^{n+1}-(d-1)\lambda^1_{f, \mbf z}>0$, i.e., that $\lambda^1_{f, \mbf z}<(d-1)^n$.

The inequality ($*$) becomes
$$
\frac{(d-1)^{n+1}-(d-1)\lambda^1_{f, \mbf z}}{(d-1)^n}\big(1+\lambda^1_{f, \mbf z}\big)\ < \ d-1. \leqno{(*)}
$$
Therefore,
$$
\big[(d-1)^{n+1}-(d-1)\lambda^1_{f, \mbf z}\big]\big(1+\lambda^1_{f, \mbf z}\big)<(d-1)^{n+1},
$$
and so
$$
(d-1)^{n+1}-(d-1)\lambda^1_{f, \mbf z}+\big[(d-1)^{n+1}-(d-1)\lambda^1_{f, \mbf z}\big]\lambda^1_{f, \mbf z}<(d-1)^{n+1}.
$$

\medskip

Subtracting $(d-1)^{n+1}$ from each side, and using that $\lambda^1_{f, \mbf z}>0$, we conclude that
$$
-(d-1)+(d-1)^{n+1}-(d-1)\lambda^1_{f, \mbf z}<0,
$$
i.e., that $(d-1)^n-1<\lambda^1_{f, \mbf z}$.

But now, we have $(d-1)^n-1<\lambda^1_{f, \mbf z}<(d-1)^n$, which is a contradiction (since $\lambda^1_{f, \mbf z}$ is an integer).
\end{proof}

\medskip

Another class of functions which does {\bf not} help with \quesref{ques:otherinequal} is given by the next proposition, which we do not prove, but which is not difficult.

\smallskip

\begin{prop} Suppose that $f=f(x,y)$ is not reduced or that $f(x,y, z)=z^p+g(x,y)$, where $p\geq 2$ and $g$ is not reduced (such an $f$ is called the $(p-1)$-fold suspension of $g$). Then,
$$
\lambda^0_{f, \mbf z}(\0)\geq \mu(f_{|_H}),
$$
where $H$ is a generic hyperplane through the origin.
\end{prop}

\bigskip

Another question is: {\bf does \thmref{thm:mainone} help us resolve Bobadilla's Conjecture?}  

\medskip

We have written about Bobadilla's Conjecture twice before, in \cite{newconjnewinv} and \cite{heplermassey}.  The conjecture is:

\smallskip

\begin{conj} \textnormal{(Bobadilla)} Suppose that $\Sigma f$ is an irreducible curve at $\0$ and that the cohomology of the Milnor fiber $F_{f, \0}$ of $f$ at $\0$ is isomorphic to the cohomology of the Milnor fiber $F_{f, \mbf p}$ for $\mbf p\in \Sigma f$ near $\0$. Then $\Sigma f$ is smooth at $\0$.
\end{conj}

In the setting of the conjecture, for $\mbf p\in\Sigma f\backslash\{\0\}$ near $\0$, the reduced cohomology of the Milnor fiber $\widetilde H^*(F_{f,\mbf p};\Z)$ is zero in all degrees other than degree $n-1$, and $\widetilde H^{n-1}(F_{f,\mbf p};\Z)\cong \Z^{\mu^\circ}$, where $\mu^\circ$ is the Milnor number of $f$ restricted to a generic hyperplane slice through $\mbf p$. At $\0$, the generic value of $\lambda^1_f$ is $\mu^\circ(\mult_\0\Sigma f)$. If $\widetilde H^*(F_{f,\mbf 0})$ is zero in all degrees other than degree $n-1$, and $\widetilde H^{n-1}(F_{f,\0};\Z)\cong \Z^{\mu^\circ}$, then we would have to have that $\lambda^0_f=\mu^\circ(-1+\mult_\0\Sigma f)$. Furthermore, $\mu(f^{[n]})=\gamma^1_f+\lambda^1_f=\gamma^1_f+\mu^\circ(\mult_\0\Sigma f)$.

We had hoped that all of this, combined with  \thmref{thm:mainone}, would give us something interesting or maybe even answer the conjecture in the affirmative. Unfortunately, we do not see how  \thmref{thm:mainone} helps.

\bibliographystyle{plain}

\bibliography{Masseybib}

\end{document}